\documentclass[12pt]{article}

\usepackage{amsmath,amssymb,amsthm,amsfonts,euscript,enumerate,latexsym}
\usepackage{tikz-cd}

\newtheorem{thm}{Theorem}[section]

\newtheorem{lem}[thm]{Lemma}
\newtheorem{prop}[thm]{Proposition}

\newcommand{\pd}{{\partial}}
\newcommand{\bpd}{{\bar\partial}}

\newcommand{\R}{Riemannian }
\newcommand{\Rm}{Riemannian manifold}

\newcommand{\K}{K\"ahler }
\newcommand{\m}{manifold}
\newcommand{\hk}{hyperk\"ahler }
\newcommand{\ma}{Monge-Amp\`ere }

\begin{document}
\baselineskip=16pt
\title{\ma equation, \hk structure and adapted complex structure} 
\author{ Su-Jen Kan}

\date{  April. 26, 2024. }

\maketitle

\begin{abstract}

In the tangent bundle of $(M,g)$, it is well-known that the \ma equation $(\pd\bar\pd \sqrt\rho)^n=0$ 
has the asymptotic expansion
$
\rho(x+iy)=\sum_{ij} g_{ij} (x) y_{i} y_{j} + O(y^4)$
near $M$.
Those 4th order terms are made explicit in this article:
$$\rho(x+iy)=\sum_{i}y_{i}^2-\frac 13\sum_{pqij} R_{i p j q}(0)x_p x_q y_{i}y_{j}+O(5).$$
At $M$, sectional curvatures  of the \K metric  $2i\pd\bar\pd\rho$  can be computed. This has enabled us to   find a family of  \K\m s whose tangent bundles have admitted complete \hk structures whereas the  adapted complex structure can only be  partially defined on the tangent bundles.

 In these cases, the study of the adapted complex structure is equivalent to the study of some gauge transformations on the baby Nahm's equation $\dot T_1+[T_0,T_1]=0.$
\end{abstract}
 \vskip 0.2truein 
\setcounter{equation}{0}
\setcounter{section}{0}
\section{Introduction} 
\setcounter{equation}{0}
For an open $\frak D_{M}\subset TM$ containing $M$, the  adapted complex structure  is the unique  complex structure  in $\frak D_{M}$
turning each leaf $\gamma^{\Bbb C}\cap \frak D_{M}$
into  a holomorphic curve. It is equivalent to the solving of the \ma equation $(\pd\bar\pd\sqrt\rho)^n=0$ on $\frak D_{M}-M$ subjected to the initial condition that the restriction of the  \K metric $2i\pd\bar\pd\rho$ to $M$ has coincided with the original \R metric $g$ on $M$. Near $M$, the asymptotic expansion
$
\rho(x+iy)=\sum_{ij} g_{ij} (x) y_{i} y_{j} + O(y^4)$
is well-known  for quite a long while. 
Any attempt to reach the 4th order terms will encounter tedious computations on complicated curvature  tensors. As far as we know, there is no further development on this yet.
In Theorem \ref{I3}, we complete the  expansion up to 4th order terms
$$\rho(x+iy)=\sum_{i}y_{i}^2-\frac 13\sum_{pqij} R_{i p j q}(0)x_p x_q y_{i}y_{j}+O(5).$$ 

Despite  the vanishing of the holomorphic sectional curvature along any Riemann foliation $\gamma^{\Bbb C},$ no other sectional curvature is known.

With these 4th order terms at hand, we are able  to compute any sectional curvature of the 
\K metric $2i\pd\bar\pd\rho$ at $M$. It turns out that if 
$(M,g)$ is positively curved   some sectional curvature of $2i\pd\bar\pd\rho$ at $M$ must be negative. 

If a non-flat $(M, g)$ has non-negative sectional curvature, an immediate application of the above result is that 
  the adapted complex structure cannot be globally defined on
 the whole tangent bundle of the \K\m\ $(\frak D_{M}, 2i\pd\bar\pd\rho).$

A \hk structure in a  differentiable
manifold $M$  consists of two anti-commuting complex structures $I,J$ and a \R metric $\kappa$, which is K\"ahlerian with respect to all $I, J$ and $IJ$.

It is natural to expect the tangent bundle of a \K\m\ might be able to produce an extra complex structure anti-commuting with the original one to produce such a \hk structure. A local existence was proved in \cite {Fe} and in \cite {Kal}. 
In other words,
 near the zero section of the tangent bundle,
   the adapted complex structure  and  \hk structures have existed simultaneously provided the underlying \K\m\ is either compact or homogeneous.
For a compact  Hermitian symmetric  $M$,
 both the \hk   and the adapted complex structures  have existed on the whole $TM$. It is thus interesting to see examples where one of the structure has existed but not the other one.
 
 We show that the  \hk structure on $TG^{\Bbb C}/H^{\Bbb C}$ constructed in \cite{Da-Sw}  can be viewed as a \hk structure on the tangent bundle of    $(T(G/H, g_n),2i\pd\bar\pd\rho),$ the \K\m \ obtained from doing the adapted complexification on the  tangent bundle of the \R\m\ $(G/H, g_n)$. However, the adapted complex structure cannot exist on the whole tangent bundle of $(T(G/H, g_n),2i\pd\bar\pd\rho)$.

 We briefly review the adapted complexification   and prove the existence of it on $T(G/H, g_n)$
 in the coming section.  In \S 3, the 4th order terms of the \ma equation $(\pd\bpd\sqrt\rho)^n=0$ are computed and are used  to decide the sign of  sectional curvatures at the zero section in \S4.  The last section is devoted to those \hk metrics on $TG^{\Bbb C}/H^{\Bbb C}$ obtained from solving  Nahm's equations. It turns out the adapted complexification on $T(G/H, g_n)$ is reduced to the solving of the baby Nahm's equation $\dot T_1+[T_0,T_1]=0.$
 
\vskip 0.2truein

\section{Adapted complex structure $TG/H$} 

\setcounter{equation}{0}

Let  $\gamma:\Bbb R\to M$ be a geodesic in a complete real-analytic \Rm\ $(M,g)$ parametrized by the arclength. We denote its complexification as 
the immersion 
 \begin{equation}
\gamma^{\Bbb C}:\Bbb C\to TM;\ \ 
 \gamma^{\Bbb C}( t+is) = (\gamma (t), s\dot \gamma(t))\in T_{\gamma (t)}M.
 \end{equation}

The adapted complex structure,  first established in [Gu-Sten] and in [Le-Sz],
 is the unique complex structure $J$ in a   domain $\frak D_{M}\subset TM$ containing $M$ such that,  for any geodesic $\gamma$,
the immersion  $\gamma^{\Bbb C}:\  \Bbb C\cap {\gamma^{\Bbb C}}^{-1}( \frak D_{M})  \to (\frak D_{M},  J)$ 
is  holomorphic. In other words, it is the unique  complex structure  in $\frak D_{M}$
turning each leaf of the  foliation $\cup_{\gamma} (\gamma^{\Bbb C}\cap \frak D_{M})$ into  a holomorphic curve.

For generic $(M, g)$, the domain $\frak D_{M}$  is properly contained in  $TM$. One necessary condition for the adapted complex structure to exist on the whole tangent bundle 
is that $(M, g)$ is
non-negatively curved. Since any bi-invariant metric of a compact Lie group $G$ has positive sectional curvature, it is natural to invest the definition on the tangent bundle of a compact Lie group.
The following construction has borrowed 
 ideas from [Sz1,2].

The tangent bundle $TG$ of a Lie group is identified with $G\times\frak g$ through the mapping $(a,v)\overset \phi\longrightarrow (a,(L_{a^{-1}})_* v)$ where $L_b$ is the left translation by $b$. For compact $G$, any bi-invariant metric $\langle , \rangle$ can be picked up to make the following $\varphi$ \begin{equation} \label{L1200}
\varphi: G\times\frak g\longrightarrow   G^{\Bbb C};  \  (g,\eta)\to g \exp i\eta
\end{equation}
a diffeomorphism. That is, the tangent  bundle $TG$ can be identified with $G^{\Bbb C}$ through $\varphi\cdot\phi$.

 A geodesic in $(G, \langle , \rangle)$ is of the form $\gamma(t)=a\exp tX$ for some $a\in G, X\in \frak g$ and $\phi (\gamma(t), s\cdot\gamma(t))=(\gamma(t), sX)\in G\times\frak g$.

 The mapping
$\Bbb C  \overset \eta\longrightarrow TG \overset \phi\longrightarrow  G\times\frak g\overset \varphi\longrightarrow   G^{\Bbb C}$ defined by 

\begin{equation}\begin{aligned}\label{3}
\varphi\cdot\phi\cdot\eta (t+is)&=\varphi\cdot\phi (\gamma(t), s\dot\gamma(t))\\
&=\varphi (\gamma(t), s X )\\
  &=a\exp tX\exp isX\\
 &=a\exp (t+is)X
 \end{aligned}
 \end{equation}  
 is holomorphic. 
 Thus, $G^{\Bbb C}$ is the adapted complexification on the tangent bundle of $(G, \langle , \rangle).$
The  adapted complex structure is defined on the whole  $T(G,g)$ for any bi-invariant metric $g$ in a compact Lie group $G$
 and the resulting complex \m\ is $G^{\Bbb C}$.

 The next step is to extend the above  from  compact Lie groups to  compact homogeneous spaces.
Let $H$ be a closed subgroup of a compact Lie group $G$, a bi-invariant metric $g$  will decompose 
$\frak g=\frak h\oplus\frak m$ where $\frak h$ is the Lie algebra of $H$ and $\frak m=T_{[e]}G/H$. The homogeneous space $G/H$  has inherited the induced metric $g_n$, a normal metric,  and $TG/H$ is diffeomorphic to 
$G\times_H\frak m$ through  the mapping $(a,v)\overset \phi\longrightarrow (a,(L_{a^{-1}})_* v)$ for any $a\in G/H$ and $v\in T_{[a]}G/H$ where the left translation $L_a: G/H\to G/H; xH\to axH$ and ${L_a}_*: T_{[x]}G/H\to T_{[ax]}G/H.$

Any geodesic of $(G/H, g_n)$ is of the form $\gamma(t)=a\exp tY, a\in G, Y\in \frak m$, p.192 Cor. 2.5 \cite {Ko-No}. $(\gamma(t),s\dot \gamma(t))\in T_{\gamma(t)}G/H$. 
We compute 
\begin{equation}\begin{aligned}\label{L111}
(L_{{\gamma(t)}^{-1}})_* (s\dot \gamma(t))&=s(L_{{\gamma(t)}^{-1}})_*\frac {d}{d\sigma}|_{\sigma=t}\gamma(\sigma)\\
&=s\frac {d}{d\sigma}|_{\sigma=t}{\gamma(t)}^{-1}\gamma(\sigma)\\
&=s\frac {d}{d\sigma}|_{\sigma=t}\exp (\sigma-t)Y\\
&=sY.
\end{aligned}
 \end{equation}  

 As in the $G^{\Bbb C}$ case,
 to see  the adapted complex structure  on $T(G/H, g_n)$,
  a  crucial step is to set a diffeomorphism between $G\times_H\frak m$ and 
$G^{\Bbb C}/H^{\Bbb C}$.
\

For this to work, we need a  decomposition theorem for $G^{\Bbb C}$. For a compact semi-simple  $G$ and a closed subgroup $H<G$, the Lie algebra $\frak g$ can be decomposed as   $\frak g=\frak h\oplus m$ with respect to the Killing metric. Mostow's decomposition theorem, \cite{Mo}, has implied 
$ G^{\Bbb C}=G\cdot\exp {im} \cdot H^{\Bbb C}.$
 One  crucial step in the proof is that the non-degenerate Killing form $B$ of $\frak g^{\Bbb C}$  is $Aut(\frak g^{\Bbb C})-$invariant which has implied $B(\frak g, i\frak g)=0$ to arrange a good basis to achieve the decomposition.
 
 If $G$ is compact but not semi-simple,  $H<G$ is a closed subgroup,  a bi-invariant metric on $G$ can decompose  $\frak g=\frak h\oplus m$ as well. However, it is not clear how to build up a non-degenerate scalar product $\langle\ ,\  \rangle$ in $\frak g^{\Bbb C}$ to play the role of the above Killing form $B$ so that $\langle \frak g , i\frak g \rangle=0$ to extend the decomposition to a generic compact Lie group $G$.
 
 Nevertheless, as a byproduct of their work about Lie group actions on complex spaces,
 Heinzner-Schwarz, p.212  Remark 9.4 \cite {He-Sch}, have confirmed the     decomposition 
 \begin{equation}\label{1}
 G^{\Bbb C}=G\cdot\exp {im} \cdot H^{\Bbb C}
  \end{equation}
 has also worked for generic compact Lie group $G$. 
It follows,  the mapping  
\begin{equation}\label{2}
 \psi: G\times_{H}\frak m\to G^{\Bbb C}/H^{\Bbb C};\  \   (g,v)\to g\exp (iv) H^{\Bbb C},
   \end{equation}  
is a $G$-equivariant diffeomorphism.
 The mapping
$\Bbb C  \overset \eta\longrightarrow TG/H \overset \phi\longrightarrow  G\times_H\frak m\overset \psi\longrightarrow   G^{\Bbb C}/H^{\Bbb C}$ defined by 

\begin{equation}\begin{aligned}\label{L113}
\psi\cdot\phi\cdot\eta (t+is)&=\psi\cdot\phi (\gamma(t), s\dot\gamma(t))\\
&=\psi (\gamma(t), sY )\\
  &=a\exp sY\exp itY\\
 &=a\exp (s+it)Y
 \end{aligned}
 \end{equation}  
 is holomorphic. Thus, $G^{\Bbb C}/H^{\Bbb C}$ is the adapted complexification on the tangent bundle of $(G/H, g_n).$

We summarize the above as a proposition.
\begin{prop}\label{L8}
Let $G$ be a compact Lie group, $H<G$ be a closed subgroup and $g_n$ be a normal metric on $G/H$. The diffeomorphism \eqref{2} has  pulled back the complex structure of $G^{\Bbb C}/H^{\Bbb C}$ to the 
adapted complex structure  on the whole $T(G/H, g_n)$.
\end{prop}

\vskip 0.2truein
\section{4-th order terms in $\rho$} 
\setcounter{equation}{0}
 The adapted complexification $(\frak D_{M},  J)$ is a Bruhat-Whitney type of complexification of $M$ in the sense that $M$ can be embedded as  a totally real submanifold of $(\frak D_{M},  J)$. Such kind of complexifications is unique up to biholomorphisms near $M$. That is, for two complexifications $\Omega_j$ of  $M$ and embeddings $i_j: M\to \Omega_j, j=1,2$, there exist neighborhoods $\Theta_j$ of $M$ in $\Omega_j$ and a biholomorphism $\phi:\Theta_1\to \Theta_2$ such that $i_2=\phi\cdot i_1$.

In \cite {Gu-St}, the authors have treated this problem by solving $(\pd\bar\pd \sqrt\rho)^n=0$.  In $\frak D_{M}$,    the \ma equation 
$(\pd\bar\pd \sqrt\rho)^n=0$ has a solution unique up to a scalar. The preferred initial condition is that the restriction of  the \K metric $h$, induced from the \K form $2\sqrt{-1}\pd\bar\pd\rho$, to $M$ has coincided with the original \R metric $g$ of $M$. In other words, for any vector fields $X$ and $Y$ in $M$,
\begin{equation}\label{e200}
h (X,Y)=2\sqrt{-1}\pd\bar\pd\rho(X,JY)=g(X,Y).
\end{equation}
This metric is called  the {\it canonical \K metric}  of the adapted complexification over $(M, g).$
The solution $\rho$ is strictly plurisubharmonic and \begin{equation}\begin{aligned}\label{e88}
\rho(p,v)=\frac {|v|^2}2
\end{aligned}\end{equation}
for any tangent vector $(p,v)\in \frak D_{M}, v\in T_pM$, {\it cf.} p.693 \cite{Le-Sz}.

 Near the zero section $M$, the solution is rather concrete in terms of local coordinates. Let $p\in M$ and let
 $\{x_1,\cdots,x_n\}$ be a geodesic normal coordinate  for  a neighborhood $U$ of $p$ in $ M$ such that $TU\simeq  U\times \Bbb R^n$ is trivial.
 Taking   $\epsilon>0$ small enough so that $T^{\epsilon}U\simeq  U\times \Bbb R_{\epsilon}^n$ for an open subset 
 $\Bbb R_{\epsilon}^n$ of $\Bbb R^n$ 
 is a Bruhat-Whitney complexification 
 which is biholomorphic to the adapted complexification $(T^{\epsilon}U, J)$.
 
  A holomorphic coordinates $\{z_j=x_j+\sqrt{-1}y_j: j=1,\cdots, n\}$ can thus be formed for the complex manifold $(T^{\epsilon}U, J)$ where $\{x_1,\cdots,x_n\}$ is a geodesic normal coordinates for $U$ and $\{y_1,\cdots, y_n\}$ is a coordinate system of $\Bbb R_{\epsilon}^n$ with $J(\frac {\pd}{\pd x_j})=\frac {\pd}{\pd y_j}$.

  Solving  the defining \ma equation in $T^{\epsilon}U$ subjected to the initial condition \eqref{e200},  the solution $\rho$ has the  asymptotic  expansion near $M$, Prop.3.1 in \cite {Kan},
 \begin{equation}\label{e1}
\rho(x+\sqrt{-1} y)=\sum_{ij} g_{ij} (x) y_{i} y_{j} + O(y^4)
\end{equation}
where the higher order terms have depended upon the curvature behavior of the metric $g$.  Locally, the \K potential $\rho$ has agreed with the Euclidean potential up to third order terms. Thus,
$(M,g)$ is a complete totally geodesic submanifold of  $(\frak D_{M},h)$.
Throughout this article, the notation $O(y^4)$ in \eqref{e1} means $$\lim_{|y|\to 0}\frac {\rho(x+\sqrt{-1} y)-\sum_{ij} g_{ij} (x) y_{i} y_{j}}{|y|^4}<\infty.$$

\

In a \Rm\ $(M,g)$ with the Levi-Civita connection $\triangledown$, the curvature tensor $R$ is defined, for any vector fields $X,Y,Z,W$, to be

\begin{equation}\begin{aligned}\label{e10}
R(X,Y,Z,W)&=g( \triangledown_X\triangledown_Y W- \triangledown_Y\triangledown_X W- \triangledown_{[X,Y]} W,Z)\\
:&=g( R(X,Y) W,Z).\\
\end{aligned}\end{equation}
The sectional curvature of the plan $\Pi$ spanned by the  vectors $X$ and $Y$ is defined as 
\begin{equation}\begin{aligned}
S(\Pi)=\frac {R(X,Y,X,Y)}{\|X\wedge Y\|^2}
\end{aligned}\end{equation}
which is independent of any choice of two independent  vectors in the plane $\Pi$.

In a local coordinate chart $\{x_j\}_{ j=1}^n$, the following notation is commonly adapted:
\begin{equation}\begin{aligned}\label{e11}
R_{ijkl}(p):= R(\frac {\pd}{\pd x_i},\frac {\pd}{\pd x_j},\frac {\pd}{\pd x_k},\frac {\pd}{\pd x_l})(p).
\end{aligned}\end{equation}

Furthermore, if  $\{x_j\}_{ j=1}^n$ is 
 a  geodesic normal coordinate centered at 0, the  following asymptotic expansion works for small $x$,
\begin{equation}
g_{ij}(x)=\delta_{ij}-\frac 13\sum_{p,q} R_{i p j q}(0)x_p x_q +O(x^3).
\end{equation}
Substituting this into \eqref{e1}, $\rho$ has the following asymptotic expansion:
 \begin{equation}\begin{aligned}\label{e2}
\rho(x+\sqrt{-1} y)&=\sum_{ij} g_{ij} (x) y_{i} y_{j} + O(y^4)\\
&=\sum_{i}y_{i}^2-\frac 13\sum_{ij p q} R_{i p j q}(0)x_p x_q y_{i}y_{j}+\sum_{ij}O(x^3)y_{i}y_{j}+O(y^4)\\
&=\sum_{i}y_{i}^2-\frac 13\sum_{ij p q} R_{i p j q}(0)x_p x_q y_{i}y_{j}+\sum_{i\le j\le\gamma\le\delta} \mathcal A_{ij\gamma\delta} y_{i}y_{j}y_{\gamma}y_{\delta} +O(5),
\end{aligned}\end{equation}
where $O(5)$ is the sum of polynomials of orders, in $x$ and $y$, greater than or equal to 5.
The aim of this section is to show
 those constant $\mathcal A_{ij\gamma\delta}$ have vanished. We'll need a general lemma on the inverse matrix to start with.
 \begin{lem}\label{l1}
Let $A=\left [ \delta_{ij}+a_{ij}(y)\right ]$ be a $n\times n$ matrix  where $a_{ij}(y)$ is a homogeneous polynomial of degree 2 in $y=(y_1,\cdots, y_n).$
Then $A^{-1}= \left [ b^{ij}(y)\right ]$
with $$  b^{ij}(y)=\left\{ \begin{array} {ll}
-a_{ij}(y)+O(y^4) & \mbox {if $i\ne j$};\\
1- a_{jj}(y)+O(y^4)& \mbox {if $i= j$}.
\end{array}
\right.$$ 
\end{lem}
\begin{proof}
The determinant 
\begin{equation}\begin{aligned}
\triangle=\det A=1+\sum_{j=1}^n a_{jj}(y) +O(y^4),
\end{aligned}\end{equation}
then 
\begin{equation}\begin{aligned}
\triangle^{-1}=1-\sum_{j=1}^n a_{jj}(y) +O(y^4).
\end{aligned}\end{equation}
Let $A^{ij}$ denote the $(n-1)\times (n-1)$ matrix obtaining from deleting the $i$-th column and the $j$-th row from $A$. Denoting the inverse matrix $A^{-1}= \left [ b^{ij}(y)\right ]$, then
\begin{equation}\begin{aligned}\label{e22}
b^{ij}(y)=\frac 1{\triangle}(-1)^{i+j} \det A^{ij}.
\end{aligned}\end{equation}
$A^{jj}$ is an $(n-1)\times (n-1)$ with diagonals $a_{11}(y),..,a_{j-1,j-1}(y),a_{j+1,j+1}(y),
\cdots, a_{nn}(y)$ then 
\begin{equation}\begin{aligned}\label{e23}
\det A^{jj}&=1+\sum_{k\ne j} a_{kk} (y)+O(y^4);\\
b^{jj}(y)&=\frac 1{\triangle}(-1)^{j+j} \det A^{jj}=1-a_{jj}(y)+O(y^4).
\end{aligned}\end{equation}
For $i< j$, the components $a_{ii}(y)$ and $a_{jj}(y)$ were missing in the $(n-1)\times (n-1)$ matrix $A^{ij}$. In the computation of $\det A^{ij}$, the only possibility with order $\le 2$ comes from the term $(-1)^{j-i+1}a_{ij}(y)\Pi_{k\ne i,j} a_{kk}(y)$. Thus
\begin{equation}\begin{aligned}\label{e24}
\det A^{ij}&=(-1)^{j-i+1}a_{ij}(y) +O(y^4);\\
b^{ij}(y)&=\frac 1{\triangle}(-1)^{i+j} \det A^{ij}\\&=-(1-\sum_{j=1}^n a_{jj}(y))a_{ij}(y)+O(y^4)\\
&=-a_{ij}(y)+O(y^4).
\end{aligned}\end{equation}
The situation $i>j$ goes exactly the same line.
\end{proof} 
In \eqref{e2}, we have asked the permutation of the indexes in $\mathcal A$ is in a non-decreasing order. We will use the same notation for a general permutation, that is, $\mathcal A_{ij\gamma\delta}\ne 0$ only if $i\le j\le\gamma\le\delta$. Let's denote the following constants $B_{\alpha ijk}$ and $C_{\alpha\beta kl}$ as:
 \begin{equation}\begin{aligned}\label{e26} 
 B_{\alpha ijk}&:=\mathcal A_{\alpha ijk}+\mathcal A_{ i\alpha jk}+\mathcal A_{ ij\alpha k}+\mathcal A_{ ijk\alpha};\\
  C_{\alpha\beta kl}&:=B_{\beta\alpha kl}+B_{\beta k\alpha l}+B_{\beta k l\alpha}.  
\end{aligned}\end{equation}
\begin{lem}\label{l2}
Let $i\le j\le k\le l, \{\alpha,\beta,\gamma,\delta\}$ be a permutation of $\{i,j,k,l\}.$ Then
$$\sum_{\alpha\beta\gamma\delta\in \{i,j,k,l\}}(B_{\alpha\beta\gamma\delta}-\frac 12 
C_{\delta\gamma\alpha\beta})=-2\mathcal A_{ijkl}.$$
 \end{lem}
 \begin{proof}
 As $\alpha\beta\gamma\delta$ run through $i,j,k,l$, $\sum _{\delta\gamma\alpha\beta}C_{\delta\gamma\alpha\beta}$ will run through all $B_{\alpha\beta\gamma\delta}$ three times. Similarly, $\sum _{\delta\gamma\alpha\beta}B_{\alpha\beta\gamma\delta}$ will run through all $\mathcal A_{\alpha\beta\gamma\delta}$ four times.
 However, $\mathcal A_{\alpha\beta\gamma\delta}\ne 0$ only if $\alpha\le\beta\le\gamma\le\delta.$
 \begin{equation}\begin{aligned}
 -\frac 12 \sum_{\alpha\beta\gamma\delta}
 C_{\delta\gamma\alpha\beta}+
 \sum_{\alpha\beta\gamma\delta}B_{\alpha\beta\gamma\delta}
 &=-\frac 32\sum_{\alpha\beta\gamma\delta}B_{\alpha\beta\gamma\delta}+\sum_{\alpha\beta\gamma\delta}B_{\alpha\beta\gamma\delta}\\
 &=\frac{-1}2\sum_{\alpha\beta\gamma\delta}B_{\alpha\beta\gamma\delta}
 =-2\mathcal A_{ijkl}.
\end{aligned}\end{equation}

 \end{proof}

 The \ma equation 
$(\pd\bar\pd \sqrt\rho)^n=0$ is equivalent to
\begin{equation}\begin{aligned}\label{e19} 
 \sum_{\alpha} \rho^{\alpha} \rho_{\alpha}=2\rho
 \end{aligned}\end{equation}
where 
\begin{equation}\begin{aligned}
 \rho^{\alpha}=\sum_{\beta}  \rho^{\alpha\bar \beta}\rho_{\bar\beta}
 \end{aligned}\end{equation}
and $(\rho^{\alpha\bar\beta})$ is the inverse matrix to $(\rho_{\alpha\bar\beta})$ so that $\sum_{\beta}\rho^{\alpha\bar\beta}\rho_{\gamma\bar\beta}=\delta_{\alpha\gamma}.$ 

 We  make a computation on $\rho_{\alpha}=\frac 12 (\frac {\pd\rho}{\pd x_{\alpha}}-\sqrt {-1}\frac {\pd\rho}{\pd y_{\alpha}})$. The expression of $\rho$ in  \eqref{e2} has ensured that $\frac {\pd\rho}{\pd x_{\alpha}}(\sqrt{-1}y)$ lies in  $O(y^4)$ terms. Then
 \begin{equation}\begin{aligned}\label{e29} 
\rho_{\alpha}(\sqrt{-1}y)=\frac {\pd\rho}{\pd z_{\alpha}}(\sqrt{-1}y)
&=-\sqrt{-1}y_{\alpha}-\frac {\sqrt{-1}}{2}\sum_{jkl}B_{\alpha jkl}y_j y_ky_l +O(y^4);\\
\rho_{\bar\alpha}(\sqrt{-1}y)=\frac {\pd\rho}{\pd \bar z_{\alpha}}(\sqrt{-1}y)
&=\sqrt{-1}y_{\alpha}+\frac {\sqrt{-1}}{2}\sum_{jkl}B_{\alpha jkl}y_jy_ky_l +O(y^4). 
   \end{aligned}\end{equation}

To find the inverse matrix $(\rho^{\alpha\bar\beta})$,
we take derivatives to $\rho$ and then take the value at the point $\sqrt{-1}y$,
\begin{equation}\begin{aligned}\label{e20} 
 \frac{\pd^2\rho(\sqrt{-1}y) }{\pd z_{\alpha}\pd z_{\bar \beta}}
&=\frac 1{4}\left ((\frac {\pd^2\rho}{\pd x_{\alpha} \pd x_\beta}+\frac {\pd^2\rho}{\pd y_{\alpha}\pd y_\beta})-\sqrt {-1}(\frac {\pd^2\rho}{\pd x_{\alpha} \pd y_\beta}-\frac {\pd^2\rho}{\pd x_\beta\pd y_{\alpha}})\right )(\sqrt{-1}y).\\
 \end{aligned}\end{equation}
 The third and the fourth terms are in $O(y^3)$, so it left to check the first and the second terms.
 \begin{equation}\begin{aligned}\label{e25} 
 \frac {\pd^2\rho}{\pd x_{\alpha} \pd x_\beta}(\sqrt{-1}y)&=\frac{-1}3\sum_{ij}(R_{i\beta j\alpha}(0)+R_{i\alpha j\beta}(0))y_i y_j;\\
 \frac {\pd^2\rho}{\pd y_{\alpha} \pd y_\beta}(\sqrt{-1}y)&=2\delta_{\alpha\beta}+\frac {\pd}{\pd y_{\alpha}}\sum_{l\gamma\delta}B_{\beta l\gamma\delta} y_ly_{\gamma}y_{\delta}\\
 &=2\delta_{\alpha\beta}+\sum_{k,l} C_{\alpha \beta kl}y_k y_l.
 \end{aligned}\end{equation}
 Plugging \eqref{e25} to \eqref{e20}, we have
 \begin{equation}\begin{aligned}\label{e27} 
 2\rho_{\alpha\bar\beta}(\sqrt{-1}y)=\delta_{\alpha\beta}+ D_{\alpha\beta}(y)+O(y^3)
 \end{aligned}\end{equation}
 where 
 \begin{equation}\begin{aligned}\label{e28} 
 D_{\alpha\beta}(y):=\sum_{i, j} \left (\frac 12 C_{\alpha\beta ij}-\frac 16(R_{i\beta j\alpha}(0) +R_{i\alpha j\beta}(0))\right )y_i y_j
 \end{aligned}\end{equation}
 is a homogeneous polynomial of degree 2 in $y$ and the $O(y^3)$ term is the sum of polynomials   in $y$ of degree $\ge 3$.
To find the inverse matrix up to order 2, we may throw away those $O(y)$ terms away. So, the matrix $\left [ 2\rho_{\alpha\bar\beta}\right ]$ has fit the form in Lemma \ref {l1}. Denoting the inverse matrix as $\left [ 2\rho_{\alpha\bar\beta}(\sqrt{-1}y)\right ]^{-1}:=\left [ G^{\alpha\bar\beta}(y)\right ].$ By Lemma \ref{l1},
$$  G^{\alpha\bar\beta}(y)=\left\{ \begin{array} {ll}
-D_{\alpha\beta}(y)+O(y^4) & \mbox {if $\alpha\ne \beta$};\\
1- D_{\alpha\beta}(y)+O(y^4)& \mbox {if $\alpha = \beta$}.
\end{array}
\right.$$ 
Then 
\begin{equation}\begin{aligned}\label{e30} 
 \rho^{\alpha\bar\beta}(\sqrt{-1}y)=2\delta_{\alpha\beta}-2D_{\alpha\beta}(y)+O(y^4). \end{aligned}\end{equation}

\begin{lem}\label{l5}
$$\mathcal A_{ijkl}=\frac 1{18}\sum_{\alpha\beta\gamma\delta \in \{i,j,k,l\}}(R_{\alpha\gamma\beta\delta}(0)+R_{\alpha\delta\beta\gamma}(0)),\ i\le\j\le k\le l.$$
\end{lem}
\begin{proof}
The idea is to equating those  4th order terms in both sides of 
 \begin{equation}\begin{aligned}\label{e31}
   \sum_{\alpha\beta} \rho^{\alpha\bar\beta}(\sqrt{-1}y)\rho_{\bar\beta}(\sqrt{-1}y)\rho_{\alpha}(\sqrt{-1}y)=2\rho(\sqrt{-1}y).
 \end{aligned}\end{equation}
 By \eqref{e30} and \eqref{e29}, 4th order terms of the left hand side come from
  \begin{equation}\begin{aligned}\label{e32}
   \sum_{\alpha\beta} (2\delta_{\alpha\beta}-2D_{\alpha\beta}(y))(\sqrt{-1}y_{\beta}+\frac {\sqrt{-1}}{2}\sum_{jkl}B_{\beta jkl}y_j y_ky_l )(-\sqrt{-1}y_{\alpha}-\frac {\sqrt{-1}}{2}\sum_{jkl}B_{\alpha jkl}y_j y_ky_l ). 
    \end{aligned}\end{equation}
They are, by \eqref{e28},
  \begin{equation}\begin{aligned}\label{e33}
&2\sum_{\alpha jkl}B_{\alpha jkl}y_{\alpha}y_jy_ky_l-2\sum_{\alpha\beta}D_{\alpha\beta}(y)y_{\alpha}y_{\beta}\\
&=2\sum_{\alpha jkl}B_{\alpha jkl}y_{\alpha}y_jy_ky_l-2\sum_{\alpha\beta ij}\left (\frac 12 C_{\alpha\beta ij}-\frac 16(R_{i\beta j\alpha}(0) +R_{i\alpha j\beta}(0))\right )y_iy_j y_{\alpha}y_{\beta}\\
&=2\sum_{\alpha \beta\gamma\delta}\left (B_{\alpha \beta\gamma\delta}-\frac 12 C_{\delta\gamma\alpha\beta}+\frac 16(R_{\alpha\delta\beta \gamma}(0) +R_{\alpha \gamma\beta\delta}(0))\right ) y_{\alpha}y_{\beta}y_{\gamma}y_{\delta}.
    \end{aligned}\end{equation}
Equating these 4th order terms with those 4th order terms from the right hand side, we get: for $i\le\j\le k\le l$, by Lemma \ref{l2}
\begin{equation}\begin{aligned}\label{e34} 
\mathcal A_{ijkl}&=\sum_{\alpha \beta\gamma\delta\in \{i,j,k,l\}}\left (B_{\alpha \beta\gamma\delta}-\frac 12 C_{\delta\gamma\alpha\beta}+\frac 16(R_{\alpha\delta\beta \gamma}(0) +R_{\alpha \gamma\beta\delta}(0))\right )\\ 
&=-2\mathcal A_{ijkl}+\frac 16\sum_{\alpha \beta\gamma\delta}(R_{\alpha\delta\beta \gamma}(0) +R_{\alpha \gamma\beta\delta}(0)).
\end{aligned}\end{equation}
Then 
$$\mathcal A_{ijkl}=\frac 1{18}\sum_{\alpha\beta\gamma\delta \in \{i,j,k,l\}}(R_{\alpha\gamma\beta\delta}(0)+R_{\alpha\delta\beta\gamma}(0)),\ i\le j\le k\le l.$$
\end{proof}

\begin{lem}\label{l6}
$$\mathcal A_{ijkl}=0, \forall i\le j\le k\le l.$$
\end{lem}
\begin{proof}
When we sum up all those $R_{\alpha\gamma\beta\delta}(0)$, for example, for a fixed pair of $\alpha=i,\gamma=j$, there must contain both the pair $\beta=k,\delta=l$ and the pair $\beta=l,\delta=k$. So, for any fixed $\alpha$ and $\gamma$,
$\sum_{\beta\delta}R_{\alpha\gamma\beta\delta}(0)=0.
$
\begin{equation}\begin{aligned}
\sum_{\alpha\beta\gamma\delta \in \{i,j,k,l\}}(R_{\alpha\gamma\beta\delta}(0)+R_{\alpha\delta\beta\gamma}(0))
&=\sum_{\alpha\gamma}\sum_{\beta\delta}R_{\alpha\gamma\beta\delta}(0)+
\sum_{\alpha\delta}\sum_{\beta\gamma}R_{\alpha\delta\beta\gamma}(0)=0. 
\end{aligned}\end{equation}
\end{proof}

The following Theorem has been proved.

\begin{thm}\label{I3} 
For $x$ small in a normal coordinate in $M$ centered at 0 and $y$ the corresponding small imaginary part, $\rho$ has the following asymptotic expansion:
$$\rho(x+\sqrt{-1}y)=\sum_{i}y_{i}^2-\frac 13\sum_{pqij} R_{i p j q}(0)x_p x_q y_{i}y_{j}+O(5).$$

\end{thm}

\vskip 0.2truein

\section{Estimate on some sectional curvatures} 
\setcounter{equation}{0}
The   manifold $\frak D_{M}$ is equipped with the \K metric $h$ obtained from the  \K form $\omega=-2i\pd\bar\pd\rho$ and the adapted complex structure $J$. We use $K$ to denote the curvature tensor of  this \K metric, {\it i.e.,} 
\begin{equation}\begin{aligned}\label{e12}
K(X,Y,Z,W)=h( \hat\triangledown_X\hat\triangledown_Y W- \hat\triangledown_Y\hat\triangledown_X W- \hat\triangledown_{[X,Y]} W,Z)=h (K(X,Y)W,Z)
\end{aligned}\end{equation}
where $\hat\triangledown$ is the corresponding Levi-Civita connection of the \Rm\ $(\frak D_{M},h).$
The holomorphic sectional curvature of the $J$-invariant plane $\Pi$ spanned by the  vectors $X$ and $JX$ is given as 
\begin{equation}
S(\Pi)=\frac {K(X,JX,X,JX)}{\|X\wedge JY\|^2}.
\end{equation}
From the fact that the connection is almost complex
 $K(JX,JY)=K(X,Y); K(X,Y)\cdot J=J\cdot K(X,Y)$. The sectional curvature of the plane spanned by $JX$ and $JY$ is the same as  
the sectional curvature of the plane spanned by $X$ and $Y$:
\begin{equation}\begin{aligned}
K(JX,JY,JX,JY)&=\langle K(JX,JY)JY,JX\rangle\\
&=\langle K(X,Y)JY,JX\rangle\\
&=\langle J K(X,Y)Y,JX\rangle\\
&=\langle  K(X,Y)Y,X\rangle\\
&= K(X,Y,X,Y).
\end{aligned}\end{equation}

In terms of local coordinate $z_j=x_j+\sqrt {-1} y_j$, the curvature tensors on vectors $\{\frac{\pd}{\pd z_j}, \frac{\pd}{\pd\bar z_k}: j,k=1,\cdots, n\}$
can be computed with $K_{ij\bar k l}$ defined in a natural way. For instance,
\begin{equation}\begin{aligned}
K_{i\bar j k\bar l}&:=K(\frac{\pd}{\pd z_i},\frac{\pd}{\pd \bar z_j}, \frac{\pd}{\pd z_k},\frac{\pd}{\pd \bar z_l});\\
K_{ij\bar k l}&:=K(\frac{\pd}{\pd z_i},\frac{\pd}{\pd z_j}, \frac{\pd}{\pd\bar z_k},\frac{\pd}{\pd z_l}).
\end{aligned}\end{equation}

It is well-known,  [Ko-No]  IX.5,  that  only the following types  can be different from 0:
\begin{equation}
K_{i\bar j k\bar l},\  K_{i\bar j \bar k l}, \ K_{\bar i j k\bar l},\  K_{\bar i j \bar k l},
\end{equation}
which can be interpreted by derivatives of the potential function of the \K metric. Since 
the \K form 
$\omega=2\sqrt{-1}\pd\bar\pd\rho,$

\begin{equation}\begin{aligned}\label{e3}
K_{i\bar jk\bar l}&=\frac{\pd^2\rho_{k\bar l}}{\pd z_{i}\pd z_{\bar j}}-\sum_{\nu\mu}\rho^{\nu\bar\mu}\frac{\pd\rho_{k\bar\mu}}{\pd z_{i}}\frac{\pd\rho_{\nu\bar l}}{\pd \bar z_{j}}\\
&=\frac{\pd^4\rho}{\pd z_{i}\pd z_{\bar j}\pd z_{k}\pd z_{\bar l}}-\sum_{\nu\mu}\rho^{\nu\bar\mu}\frac{\pd^3\rho}{\pd z_{i}\pd z_{ k}\pd\bar z_{\mu}}\frac{\pd^3\rho}{\pd \bar z_{j}\pd\bar z_l\pd z_{\nu}}.
\end{aligned}\end{equation}

 As
\begin{equation}
\frac{\pd}{\pd x_j}=\frac{\pd}{\pd z_j}+\frac{\pd}{\pd \bar z_j}; \  \  \frac{\pd}{\pd y_j}=\sqrt {-1} (\frac{\pd}{\pd z_j}-\frac{\pd}{\pd \bar z_j}),
\end{equation}
the sectional curvature of the plane spanned by $\frac{\pd}{\pd x_i}$ and $\frac{\pd}{\pd y_j}$ can be computed as following:

\begin{equation}\begin{aligned}\label{e6}
K(\frac{\pd}{\pd x_i},\frac{\pd}{\pd y_j}, \frac{\pd}{\pd x_i},\frac{\pd}{\pd y_j})&
=-K(\frac{\pd}{\pd z_i}+\frac{\pd}{\pd \bar z_i},\frac{\pd}{\pd z_j}-\frac{\pd}{\pd \bar z_j}, \frac{\pd}{\pd z_i}+\frac{\pd}{\pd \bar z_i},\frac{\pd}{\pd z_j}-\frac{\pd}{\pd \bar z_j})\\
&=-( K_{i\bar j i\bar j }+2K_{i\bar j j\bar i }+\overline {K_{i\bar j i\bar j }}).
\end{aligned}\end{equation}
Plugging the expression of $\rho$ in Theorem \ref{I3} it is clear to see,
when evaluated at the point $0$, the second part of \eqref {e3} has vanished since there is no third order terms in $\rho$. It is sufficient to take care of the first part.
\begin{equation}\begin{aligned}\label{e5}
\frac{\pd^4\rho }{\pd z_{i}\pd z_{\bar j}\pd z_{k}\pd z_{\bar l}}
=\frac 1{16}&\left ((\frac {\pd^2}{\pd x_k \pd x_l}+\frac {\pd^2}{\pd y_k\pd y_l})+\sqrt {-1}(\frac {\pd^2}{\pd x_k \pd y_l}-\frac {\pd^2}{\pd x_l\pd y_k})\right )\\
&\left ((\frac {\pd^2}{\pd x_i \pd x_j}+\frac {\pd^2}{\pd y_i\pd y_j})+\sqrt {-1}(\frac {\pd^2}{\pd x_i \pd y_j}-\frac {\pd^2}{\pd x_j\pd y_i})\right )\rho\\
=\frac 1{16}&\left (\frac{\pd^4\rho}{\pd x_i\pd x_j\pd x_k\pd x_l}+\frac{\pd^4\rho}{\pd x_i\pd x_j\pd y_k\pd y_l} +\frac{\pd^4\rho}{\pd x_k\pd x_l\pd y_i\pd y_j}+ \frac{\pd^4\rho}{\pd y_i\pd y_j\pd y_k\pd y_l} \right )\\
-\frac 1{16}&\left (\frac{\pd^4\rho}{\pd x_i\pd x_k\pd y_j\pd y_l} +\frac{\pd^4\rho}{\pd x_j\pd x_l\pd y_i\pd y_k}- \frac{\pd^4\rho}{\pd x_i\pd x_l\pd y_j\pd y_k}- \frac{\pd^4\rho}{\pd x_j\pd x_k\pd y_i\pd y_l}  \right )\\
+\frac {i}{16}&\left (\frac{\pd^4\rho}{\pd x_i\pd x_k\pd x_l\pd y_j} +\frac{\pd^4\rho}{\pd x_i\pd y_j\pd y_k\pd y_l}- \frac{\pd^4\rho}{\pd x_j\pd x_k\pd x_l\pd y_i}- \frac{\pd^4\rho}{\pd x_j\pd y_i\pd y_k\pd y_l}  \right )\\
+\frac {i}{16}&\left (\frac{\pd^4\rho}{\pd x_i\pd x_j\pd x_k\pd y_l} +\frac{\pd^4\rho}{\pd x_k\pd y_i\pd y_j\pd y_l}- \frac{\pd^4\rho}{\pd x_i\pd x_j\pd x_l\pd y_k}- \frac{\pd^4\rho}{\pd x_l\pd y_i\pd y_j\pd y_k}  \right ).
\end{aligned}\end{equation}
Evaluated at the point $0$, only terms of the type $\frac {\pd^4\rho}{\pd x_i \pd x_j\pd y_k \pd y_l}$ could be non-vanished.
\begin{equation}\begin{aligned}\label{e15}
\frac {\pd^4\rho (0)}{\pd x_i \pd x_j\pd y_k \pd y_l}&=\frac{-1}3(R_{kilj}+R_{kjli}+R_{likj}+R_{ljki})(0)\\
&=\frac{-2}3(R_{kilj}+R_{likj})(0).
\end{aligned}\end{equation}

\begin{equation}\begin{aligned}\label{e4}
\frac{\pd^4\rho(0)}{\pd z_{i}\pd z_{\bar j}\pd z_{k}\pd z_{\bar l}}
&=\frac 1{16}\left (\frac{\pd^4\rho(0)}{\pd x_i\pd x_j\pd y_k\pd y_l} +\frac{\pd^4\rho(0)}{\pd x_k\pd x_l\pd y_i\pd y_j} \right )\\
&\ \ -\frac 1{16}\left (\frac{\pd^4\rho(0)}{\pd x_i\pd x_k\pd y_j\pd y_l} +\frac{\pd^4\rho(0)}{\pd x_j\pd x_l\pd y_i\pd y_k}- \frac{\pd^4\rho(0)}{\pd x_i\pd x_l\pd y_j\pd y_k}- \frac{\pd^4\rho(0)}{\pd x_j\pd x_k\pd y_i\pd y_l}  \right )\\
&=\frac {-1}{24} (R_{kilj}+R_{likj}+
R_{ikjl}+R_{jkil}     )(0)\\
 &\ \  -\frac {1}{24} (-R_{jilk}-R_{lijk}   -R_{ijkl}-R_{kjil}  + R_{jikl}+ R_{kijl}      +R_{ijlk} +R_{ljik})(0)\\
 &=\frac 16(R_{ijkl}+R_{ilkj})(0).
\end{aligned}\end{equation}
Taking $k=i, l=j$ and then taking $k=j, l=i$
\begin{equation}\begin{aligned}\label{e7}
K_{i\bar j i\bar j}(0)&=\frac{\pd^4\rho(0)}{\pd z_{i}\pd z_{\bar j}\pd z_{i}\pd z_{\bar j}}\\
&=\frac {1}{6}\left (R_{ijij}+ R_{ijij}      \right )(0)\\
&=\frac {1}{3}R_{ijij}(0);
\end{aligned}\end{equation}

\begin{equation}\begin{aligned}\label{e8}
K_{i\bar j j\bar i }(0)&=\frac{\pd^4\rho(0)}{\pd z_{i}\pd z_{\bar j}\pd z_{j}\pd z_{\bar i}}\\
&=\frac {1}{6}\left (R_{ijji}+ R_{iijj}      \right )(0)\\
&=\frac {-1}{6}R_{ijij}(0)\end{aligned}\end{equation}

\begin{prop}\label{I4}

At  $p\in M$, the sectional curvature of the metric $h$ on the plane spanned by $\frac{\pd}{\pd x_i}$ and $\frac{\pd}{\pd y_j}$ is 
$\frac {-1}{3}R_{ijij}(p)$. All holomorphic sectional curvatures of the metric $h$ have vanished at $p$.
\end{prop}
\begin{proof} The sectional curvature for the plane spanned by $\frac{\pd}{\pd x_i}$ and $\frac{\pd}{\pd y_j}$  is 
 $K(\frac{\pd}{\pd x_i},\frac{\pd}{\pd y_j}, \frac{\pd}{\pd x_i},\frac{\pd}{\pd y_j}).$ Take a normal coordinate centered at $p$ as discussed in the previous section. At the point 0, applying \eqref{e6}--\eqref{e8},
\begin{equation}\begin{aligned}\label{e9}
K(\frac{\pd}{\pd x_i},\frac{\pd}{\pd y_j}, \frac{\pd}{\pd x_i},\frac{\pd}{\pd y_j})(0)
&=-( K_{i\bar j i\bar j }(0)+2K_{i\bar j j\bar i }(0)+\overline {K_{i\bar j i\bar j }}(0))\\
&=-2(\frac 13-\frac 16)R_{ijij}(0)\\
&=\frac {-1}{3}R_{ijij}(0).
\end{aligned}\end{equation}
Holomorphic sectional curvatures are sectional curvatures of planes spanned by $\frac{\pd}{\pd x_j}$ and $J\frac{\pd}{\pd x_j}=\frac{\pd}{\pd y_j}$. Taking  $i=j$,   
\begin{equation}\begin{aligned}
K(\frac{\pd}{\pd x_i},\frac{\pd}{\pd y_i}, \frac{\pd}{\pd x_i},\frac{\pd}{\pd y_i})
=\frac {-1}{3}R_{iiii}=0.
\end{aligned}\end{equation}
\end{proof}

As $(M, g)$ is a totally geodesic submanifold of $(\frak D_{M},h)$, the two corresponding  Levi-Civita connections $\triangledown$ and $\hat\triangledown$ have coincided when applying to vector fields of $M$. Thus for vector fields in $M$, the two curvature tensors computed from either connection will be the same. This can be easily checked following lines of the above computations.

\begin{equation}\begin{aligned}\label{e16}
K(\frac{\pd}{\pd x_i},\frac{\pd}{\pd x_j}, \frac{\pd}{\pd x_i},\frac{\pd}{\pd x_j})&
=K(\frac{\pd}{\pd z_i}+\frac{\pd}{\pd \bar z_i},\frac{\pd}{\pd z_j}+\frac{\pd}{\pd \bar z_j}, \frac{\pd}{\pd z_i}+\frac{\pd}{\pd \bar z_i},\frac{\pd}{\pd z_j}+\frac{\pd}{\pd \bar z_j})\\
&= K_{i\bar j i\bar j }-2K_{i\bar j j\bar i }+\overline {K_{i\bar j i\bar j }}.
\end{aligned}\end{equation}
Thus
\begin{equation}\begin{aligned}\label{e13}
K(\frac{\pd}{\pd x_i},\frac{\pd}{\pd x_j}, \frac{\pd}{\pd x_i},\frac{\pd}{\pd x_j})(0)
&=2 K_{i\bar j i\bar j }(0)
-2K_{i\bar j j\bar i }(0)
\\
&=2(\frac 13+\frac 16)R_{ijij}(0)
= R_{ijij}(0)
.
\end{aligned}\end{equation}

\

The fact that $J\frac {\pd}{\pd x_j}=\frac {\pd}{\pd y_j}$ together with the $J$-invariance of $K$ has implied that 
 \begin{equation}\begin{aligned}\label{e14}
 K(\frac{\pd}{\pd y_i},\frac{\pd}{\pd y_j}, \frac{\pd}{\pd y_k},\frac{\pd}{\pd y_l})(0)
&=K(J\frac{\pd}{\pd x_i},J\frac{\pd}{\pd x_j}, J\frac{\pd}{\pd x_k},J\frac{\pd}{\pd x_l})(0)
\\
&=K(\frac{\pd}{\pd x_i},\frac{\pd}{\pd x_j}, \frac{\pd}{\pd x_k},\frac{\pd}{\pd x_l})(0)
\\
&= R_{ijij}(0)
.
\end{aligned}\end{equation}

If the adapted complex structure is defined on the whole tangent bundle of  $(M, g)$,  the metric $g$ must be non-negatively curved, {\it cf.}  [Theorem 4.3, Le-Sz]. Although their theorem was  stated  for compact $M$ only, the arguments used in the proof have solely relied on local curvature properties  which  are irrelevant to the compactness of the \m .  We denote  the resulting complex \m\  by $T(M,g)$; the canonical \K metric  by $h$.
   The following proposition is immediate.
  
 \begin{prop}\label{I5}
    Suppose the adapted complex structure is defined on the whole tangent bundle of  a non-flat $(M, g)$ with the canonical \K metric $h$.
    Then the adapted complex structure can at most  be partially defined on the  tangent bundle of  the \R\m\ $(T(M,g), h)$.
\end{prop}  
 \begin{proof}   Since the adapted complexification is globally defined and $g$ is not flat, some sectional curvature $R_{ijij}$ must be positive.
By Proposition \ref{I4}, the metric $h$ must be negatively curved at the corresponding plane. Thus the adapted complex structure can not    be  defined on the whole tangent bundle of   $(T(M,g), h)$.

\end{proof}

\vskip 0.2truein
\section{ Hyperk\"ahler structure on tangent bundles} 
\setcounter{equation}{0}
A \hk manifold, a \m\ modeled on the quaternions, is a $4n$-dimensional Riemannian manifold $(M, \kappa)$ equipped with anti-commutative complex structures $I$ and $J$  so that the metric $\kappa$ is tri-K\"ahlerian with respect to $I,J$ and $IJ$. A \hk metric is naturally  Ricci-flat and real-analytic.

Through the adapted complexification  the tangent bundle of a  \R\m\ can be partially complexified to a \K\m, 
it is expected that the tangent(cotangent) bundle of a \K\m\ may be quaternionised to a \hk\m, at least  near the zero section.
A local existence theorem has been proved  in \cite {Fe} and  in \cite {Kal}: there exists
a \hk structure  in a neighborhood of the zero section in the cotangent(tangent) bundle of a real-analytic \K \m.  

For a compact  Hermitian symmetric  space $M$,
 both the \hk   and the adapted complex structures  have existed in the whole $TM$, {\it cf.} \cite{Da-Sz}.

 For a generalized flag \m\ $G/K$, {\it i.e.,} $G$ is compact semi-simple, $K<G$ is the centralizer of some $\eta\in\frak g$, a family of complete \hk structures on the coadjoint orbit $G^{\Bbb C}/K^{\Bbb C}=T(G/K,g_n)$ have been successfully
 constructed in
  \cite {Kr2}\ \cite {Kov} and  \cite {Biq}. That is to say, one of the complex structure comes from the adapted complexification on the tangent bundle of  $(G/K,g_n)$. 
  The \hk metric $\kappa$ is $G$-invariant and $\kappa|_{G/K}$ is  
 $G$-invariant  in the homogeneous space $G/K$.
 If $G/K$ is isotropy irreducible, 
$\kappa|_{G/K}$ is a normal metric and both the \hk and adapted complex structures have existed on the total space of the tangent bundle of the \Rm\ $(G/K, g_n)$. 

A natural question arisen is: will the two structures exist simultaneously?
It is attempting to find examples where one of the structure has existed but not the other one.
We'll work on  some global \hk structures on  tangent bundles, obtained from a \hk quotient from studying
Nahm's equations. We show at Theorem \ref{L9} that none of those tangent bundles  can possess a  globally defined 
 adapted complex structure. 

\vskip 0.2truein
\subsection{ Hyperk\"ahler structure in $TG^{\Bbb C}$. }
The coming two   sections have basically followed from \cite {Kr1} and \cite {Da-Sw} with some mirror changes in the notations.  Given a compact Lie group $G$ with Lie algebra $\frak g$, let's denote the associate infinite-dimensional quaternion $\mathcal A$   as 
\begin{equation}\label{e5}
\mathcal A=\{T_0+iT_1+jT_2+kT_3; \ T_l:[0,1]\overset {C^{\infty}}\longrightarrow \frak g, \forall l\}.
    \end{equation}  
$\mathcal A$ is an infinite-dimensional vector space endowed with three anti-commuting complex structures $I,J,K$ given by left multiplications by $i, j, k$ respectively. 
An $Ad_G$-invariant inner product 
$
\langle \ ,\ \rangle$
in $\frak g$ is specified to produce a natural $L^2$-metric in $\cal A$ defined as, for  tangent vectors $(X_0, X_1,X_2,X_3)$ and $(Y_0,Y_1,Y_2,Y_3)$ of $\cal A$\begin{equation}\label{e10}
\langle (X_0, X_1,X_2,X_3),(Y_0,Y_1,Y_2,Y_3)\rangle:=\int_0^1\sum_{j=0}^3\langle X_j(t),Y_j(t)\rangle dt.
\end{equation} 

This metric together with the complex structures $I,J,K$ has endowed $\cal A$ with a \hk structure. The three relative \K forms are 
\begin{equation}\begin{aligned}\label{e11}
\omega_I=dT_0 dT_1 + dT_2 dT_3;  \text { and cyclically.}
\end{aligned}\end{equation} 
 The function 
\begin{equation}\begin{aligned}\label{e35}
f(T):=\frac 12 T_1^2+\frac 14 T_2^2+\frac 14 T_3^2 \end{aligned}  \end{equation} 
 is a \K potential for  $\omega_I$ since
 \begin{equation}\begin{aligned}\label{e40}
2i\pd\bpd f=d Id f&=dI (T_1 dT_1+\frac 12 T_2 dT_2+\frac 12 T_3 dT_3)\\
&=d(-T_1 dT_0+\frac 12 T_2 dT_3-\frac 12 T_3 dT_2)\\
&=  dT_0\wedge dT_1+dT_2\wedge dT_3=\omega_I.
\end{aligned}  \end{equation}

 Because
$T_0+iT_1+jT_2+kT_3=T_0+iT_1+(T_2+iT_3)j, $
 the vector space $(\mathcal A,I)$ 
 can be interpreted by 
 $
\alpha=T_0+iT_1, \beta=T_2+iT_3.$

The gauge group $
\mathcal G$ consists of smooth maps $ g :[0,1]\to G$ 
has acted on $\mathcal A$ by
\begin{equation}\begin{aligned}\label{e13}
&g\cdot T_0=gT_0 g^{-1}-\frac {dg}{dt}g^{-1};\;\
 g\cdot T_j=gT_j g^{-1}, \; j=1,2,3.
\end{aligned}  \end{equation} 
The $\mathcal G$-action has kept all $I,J,K$ and the relative \K forms.
Let   $\mathcal G_0:= \{g\in \mathcal G: g(0)=g(1)=1\}$,
the corresponding moment map of this $\mathcal G_0$-action is denoted by
 $\mu=(\mu_I, \mu_J, \mu_K)$. 
 The zero locus $\mu^{-1}(0)$ is exactly the solution space of the Nahm's equations 
 \begin{equation}\begin{aligned}\label{e14}
\frac {dT_1}{dt}+[T_0,T_1]+[T_2,T_3]=0 \;\;\text{ and cyclically.}
\end{aligned}  \end{equation} 
 The  moduli space $\mu^{-1}(0)/\mathcal G_0$  is indeed a \hk\m. Thus, given a compact Lie group $(G, \langle\ ,\ \rangle)$, a   \hk\m\ $\mu^{-1}(0)/\mathcal G_0$ is associated to it through the above construction. 
  
Fixing one complex structure  $I$ and set $\alpha=T_0+iT_1, \beta=T_2+iT_3$, the $\mathcal G_0$-action on $\mu^{-1}(0)$ can be extended to a holomorphic $\mathcal G_0^c$-action on $\mu_c^{-1}(0):=(\mu_J+i\mu_K)^{-1}(0)$ to make $\mu_c^{-1}(0)/\mathcal G_0^c$ a \K\m\ biholomorphic to $({\mu}^{-1}(0)/ \mathcal G_0,I)$.

 One   advantage of taking the  complex \m\ $\mu_c^{-1}(0)/\mathcal G_0^c$ is to develop a concrete biholomorphism with $TG^{\Bbb C}$.
\begin{equation}\begin{aligned}\label{e15}
 &({\mu}^{-1}(0)/ \mathcal G_0,I)\overset \eta\longrightarrow {\mu_c}^{-1}(0)/ \mathcal G_0^c\overset \phi \longrightarrow TG^{\Bbb C}=G^{\Bbb C}\times \frak g^c\\
&(T_0,T_1,T_2,T_3)\longrightarrow (\alpha=T_0+iT_1,\beta=T_2+iT_3)\longrightarrow (g(0)^{-1},\beta(1))
\end{aligned}\end{equation}
where $g\in \mathcal G^c$ such that $g\cdot\alpha=0$ and $ g(1)=id.$ The requirement  $g\cdot\alpha=0$ is  equivalent to solving the linear ordinary differential equation $\frac{dg}{dt}=g\alpha$ and the initial condition $g(1)=id$ will determine such a $g$ uniquely.

Restricted to the subspace $\{T_2=T_3=0\}$,  we define    
\begin{equation}\begin{aligned}\label{e16}
\varphi: {\mu}^{-1}(0)/ \mathcal G_0|_{T_2=T_3=0}&\to G\times \frak g\\
(T_0,T_1)&\to (\xi(0)^{-1},T_1(1))
\end{aligned}\end{equation}
 where $\xi$ is the unique element in $\mathcal G$ gauging $T_0$ to 0 with the initial condition 
  $ \xi(1)=id.$ 
\begin{lem}\label{L1}
$\varphi$ in \eqref{e16} is a diffeomorphism whose inverse can be explicitly written down as  
\begin{equation}\begin{aligned}\label{e100}
\varphi^{-1}: G\times \frak g&\to {\mu}^{-1}(0)/ \mathcal G_0|_{T_2=T_3=0} \\
(a,v)&\to (-\frac {dh}{dt} h^{-1}, hvh^{-1})
\end{aligned}\end{equation}
for some $h\in\mathcal G$ with the boundary condition $h(0)=a, h(1)=id.$ 
\end{lem}
\begin{proof}
Given $a\in G$, it is always possible to find an $h\in \mathcal G$ with boundary conditions $h(0)=a, h(1)=id.$
A direct computation verifies both $-\frac {dh}{dt} h^{-1}$ and $ hvh^{-1}$ are smooth mappings from $[0,1]$ to $\frak g$ satisfying \begin{equation}\begin{aligned}\label{e18}
\frac {d}{dt}(hvh^{-1})
 +[-\frac {dh}{dt} h^{-1}, hvh^{-1}]=0
\end{aligned}\end{equation}
which is the  equation 
$\dot T_1+[T_0,T_1]=0$. So, 
$(-\frac {dh}{dt} h^{-1}, hvh^{-1})\in 
{\mu}^{-1}(0)/ \mathcal G_0|_{T_2=T_3=0}.$

It remains to show the mapping is well-defined. Suppose there are two $h,  l\in \mathcal G$ sharing the same boundary conditions. Then $lh^{-1}\in \mathcal G$ as well which has gauged the point 
$(-\frac {dh}{dt} h^{-1}, hvh^{-1})$ to the point $(-\frac {dl}{dt} l^{-1}, lvl^{-1})$. In other words, they have represented the same point in the moduli space; the mapping $\varphi^{-1}$ is well-defined.
\end{proof}

 Restricted to the subspace $T_2=T_3=0$ in \eqref{e15} along with the surjectivity proved in Lemma \ref{L1},  
 the following diffeomorphisms are well-defined.
\begin{equation}\begin{aligned}\label{e17}
 G\times \frak g\overset {\varphi^{-1}}\longrightarrow&({\mu}^{-1}(0)/ \mathcal G_0|_{T_2=T_3=0},I)&&\overset \eta\longrightarrow {\mu_c}^{-1}(0)/ \mathcal G^c|_{\beta=0}\overset \phi\longrightarrow G^{\Bbb C}\\ (\xi(0)^{-1},T_1(1))\longrightarrow
&\;\;\;\;\;\;\;\;\; (T_0,T_1)&&\longrightarrow \alpha=T_0+iT_1\longrightarrow  g(0)^{-1}
\end{aligned}\end{equation}
where $\xi\in\mathcal G$ such that $\xi\cdot T_0=0$ and $\xi(1)=id; g\in \mathcal G^c$ such that $g\cdot \alpha=0$ and $g(1)=id$. When $T_2\equiv T_3\equiv 0$, the criteria for $(T_0,T_1)$ is simply the baby Nahm's equation
$
\frac {dT_1}{dt}+[T_0,T_1]=0. $ The $\mathcal G$-invariance of the equation together with the fact  $\xi\cdot T_0=0$ has implied   $\frac {d(\xi\cdot T_1)}{dt}=0$. That is, $\xi\cdot T_1$ is a constant which can be taken as 
\begin{equation}\begin{aligned}\label{e22}
(\xi\cdot T_1)(t)=\xi(1)T_1(1)\xi(1)^{-1}=T_1(1),\ \forall t\in [0,1]. 
\end{aligned}\end{equation}

\begin{lem}\label{L2}
\begin{equation}\begin{aligned}\label{e18} 
\phi\cdot\eta\cdot\varphi^{-1}: \ & G\times\frak g\to G^{\Bbb C}\\
 &(a,v)\to a\exp iv.
\end{aligned}\end{equation}
This diffeomorphism has pulled back the complex structure of $G^{\Bbb C}$ to the adapted complex structure on $T(G,\langle\ ,\ \rangle)$.
\end{lem}

\begin{proof}
By Lemma \ref{L1}, every point $(a,v)\in G\times \frak g$ is of the form $(\xi(0)^{-1},T_1(1))$ as in described in \eqref{e16}.
Take 
\begin{equation}\begin{aligned}\label{e19}
\hat g(t)=\exp{(i(t-1)T_1(1))}\xi(t): [0,1]\overset {C^{\infty}}\longrightarrow G^c.
\end{aligned}\end{equation}
Then $\hat g (1)=\xi(1)=id, \hat g\in \mathcal G^c$ and 
\begin{equation}\begin{aligned}\label{e23}
\hat g\cdot\alpha&=
\exp{(i(t-1)T_1(1))}\xi(t)\cdot (T_0+iT_1)\\
&=\exp{(i(t-1)T_1(1))}\cdot (\xi\cdot T_0+\xi\cdot iT_1)\\
& =\exp{(i(t-1)T_1(1))}\cdot (i T_1(1))\\
&=e^{(i(t-1)T_1(1))} i T_1(1)e^{(-i(t-1)T_1(1))}-\frac{d}{dt}(e^{(i(t-1)T_1(1))})e^{(-i(t-1)T_1(1))}\\
&=i T_1(1)-i T_1(1)=0.
\end{aligned}\end{equation}
Both $g$ and $\hat g$ have gauged $\alpha$ to 0 with the same initial condition at $t=1$, by the uniqueness of ODE, $\hat g\equiv g$.
\begin{equation}\begin{aligned}\label{e24}
g^{-1}(0)={\hat g}^{-1}(0)=\xi(0)^{-1}\exp(iT_1(1)).
\end{aligned}\end{equation}
Identifying $(\xi(0)^{-1},T_1(1))$ with 
$(a,v)$, the diffeomorphism is proved. The last statement is a consequence of Proposition \ref{L8} by taking $H=id$.
\end{proof}

In \eqref{e15}, the diffeomorphism $\phi\cdot\eta$ has pushed forward the \K form $\omega_I$ to a \K form $(\phi\cdot\eta)_*\omega_I$
in $G^{\Bbb C}\times \frak g^c$ determining the \hk metric $\kappa$ in $G^{\Bbb C}\times \frak g^c$. Restricted to the subspace $\{T_2=T_3=0\}$,
   $\phi\cdot\eta$ at  $\eqref{e17}$ is still a biholomorphism from $({\mu}^{-1}(0)/ \mathcal G_0|_{T_2=T_3=0},I)$ to $ G^{\Bbb C}$ sending the \K form $\hat\omega_I=dT_0 dT_1$ to the \K form $(\phi\cdot\eta)_*{\hat\omega_I}$
in $G^{\Bbb C}$  corresponding to  the \K metric $\kappa|_{G^c}$.

The  diffeomorphisms 
 in \eqref{e17}  has pulled back the \K form $(\phi\cdot\eta)_*{\hat\omega_I}$
 to the \K form $\varphi^{-1*}\hat\omega_{I}$ on  $G\times \frak g$. By \eqref{e35},
 the \K potential of  $\hat\omega_I$ is $\hat f (T_0,T_1)=\frac 12 \int_0^1 \langle T_1(t), T_1(t)\rangle dt$. Since $\langle\ ,\ \rangle$ is $Ad_G$-invariant, 
 $\langle Ad_gv, Ad_gv\rangle=\langle v, v\rangle$ for any $g \in G$. That is,
 for $h\in \mathcal G$
 \begin{equation}\begin{aligned}
  \label{e77}
 \hat f (-\frac {dh}{dt} h^{-1}, hvh^{-1})
=\frac 12\int_0^1 \langle h(t)vh(t)^{-1} , h(t)vh(t)^{-1}\rangle dt= \frac {|v|^2}2.\end{aligned}\end{equation}
  The identification at \eqref{e100} shows that the corresponding \K potential $ \tilde f$ in the $G\times \frak g$ model is, for $(a,v)\in  G\times \frak g$, $\tilde f(a,v)=\frac {|v|^2}{2}$, which is the standard \K potential at \eqref{e88}. 
  
 Thus, the corresponding \K metric is the canonical \K metric  of the adapted complexification on $T(G, \langle\ ,\ \rangle).$  We conclude the above discussions as the following.

\begin{prop}\label{L7}
Let $g$ be a bi-invariant metric in the compact Lie group $G$, and 
let $h$ denote the canonical \K metric 
of the adapted complexification on
 $T(G,g)$. 
 A complete \hk metric $\kappa$ has existed on the tangent bundle of the \K\m\ $(TG,h)$ whose restriction  to $TG$ is  $h$. 
\end{prop}

\vskip 0.2truein

 \subsection{ Hyperk\"ahler structure in $TG^{\Bbb C}/H^{\Bbb C}$. } 

Throughout this section, 
 $H< G$ is a closed subgroup of the compact Lie group $G$ and the notation  $\mathcal M$ is used to denote the \hk\m\ $\mu^{-1}(0)/\mathcal G_0$ developed in \S 5.1. The  $H$-action on $\mathcal M$ is given by the gauge group $\mathcal H=\{ g :[0,1]\to G; g(0)=id, g(1)\in H\}\subset \mathcal G$ via 
 \begin{equation}\begin{aligned}\label{e34}
g\cdot T_0=gT_0 g^{-1}-\frac {dg}{dt}g^{-1};\;\
 g\cdot T_j=gT_j g^{-1}, \; j=1,2,3.
\end{aligned}  \end{equation} 
 It is clear this  $H$-action has preserved the \hk structure in $\mathcal M$.  The notation $\pi_{\frak h} $ is used to denote the projection of $\frak g=\frak h\oplus\frak m$ to the $\frak h$ component.
 The \hk moment map $\Phi$ for the $H$-action on $\mathcal M 
 $ is, {\it cf.} Lemma 2 in \cite {Da-Sw},  
  $$\Phi(T)=(\Phi_I(T),\Phi_J(T),\Phi_K(T))=(\pi_{\frak h} T_1(1), \pi_{\frak h} T_2(1), \pi_{\frak h} T_3(1)).$$ 
  Thus, 
  $\Phi^{-1}(0)=\{(T_0,T_1,T_2,T_3)\in \mathcal M; T_j(1)\in \frak m, j=1,2,3\}$ and
$\Phi_c^{-1}(0)=\{(\alpha,\beta)\in \mathcal M; \beta(1)\in \frak m^c\}$. The \hk quotient $\Phi^{-1}(0)/\mathcal H$ has created a new \hk\m; \K forms and \K potentials are  descending from those of $\mathcal M$. The  $H$-action can be extended to a holomorphic  $H^c$-action via the gauge group $\mathcal H^c=\{ g :[0,1]\to G^{\Bbb C}; g(0)=id,  g(1)\in H^{\Bbb C}\}$.  

A  necessary condition to identify $(\Phi^{-1}(0)/\mathcal H,I)$ with $\Phi_c^{-1}(0)/\mathcal H^c$ is the stable points $\mathcal M^s$ of the $  H^c$-action must be  $\Phi_c^{-1}(0)$ where
\begin{equation}\begin{aligned}\label{e300}
\mathcal M^s:=\{T\in \mathcal M: \mathcal H^c \cdot T \cap \Phi^{-1}(0)\ne\emptyset\}.
\end{aligned}  \end{equation}

 Since $\mathcal M$ is of finite dimensional, the $H$-action is free and the potential function $f$ at \eqref{e35} is bounded above, the properness of $f$ is the crucial ingredient remained  to assure 
 the identification of $(\Phi^{-1}(0)/\mathcal H,I)$ with $\Phi_c^{-1}(0)/\mathcal H^c$, {\it cf.} p.199 \cite{May1} and Theorem 3.6 \cite{May3}. The function $F( k,X)$ listed in \cite{May2} is the \K potential of $\omega_I$ computed in the model 
$\mathcal W\subset G\times {\frak g}^3$, which is diffeomorphic to $\mathcal M$  by Theorem 3 of \cite{Da-Sw}. The corresponding potential in $\mathcal M$ is $f$ whose properness  has inherited from that of $F$ and the following biholomorphism is established:
 \begin{equation}\begin{aligned}\label{e36}
(\Phi^{-1}(0)/\mathcal H,I)\to\Phi_c^{-1}(0)/\mathcal H^c; \;\;(T_0,T_1,T_2,T_3)\to (\alpha=T_0+iT_1,\beta=T_2+iT_3).
\end{aligned}  \end{equation}
The next attempt is to interpret $\Phi_c^{-1}(0)/\mathcal H^c$ in terms of $G$ and $H$. The homogeneous space $G^{\Bbb C}/H^{\Bbb C}$ is reductive and $TG^{\Bbb C}/H^{\Bbb C}=G^{\Bbb C}\times_{H^{\Bbb C}} \frak m^c$.
Defining 
\begin{equation}\begin{aligned}\label{e37}
\phi: {\Phi_c}^{-1}(0)/ \mathcal H^c  \longrightarrow &TG^{\Bbb C}/H^{\Bbb C}=G^{\Bbb C}\times_{H^{\Bbb C}} \frak m^c\\
(\alpha,\beta) \longrightarrow &(g(0)^{-1},\beta(1))
\end{aligned}\end{equation}
where $g$ is the unique element in $\mathcal G$ with $g(1)=id$ and $g\cdot\alpha=0$. For any $\zeta\in \mathcal H^c$,
$(\zeta\cdot\beta) (1)=\zeta(1)\beta (1)\zeta(1)^{-1}=Ad_{\zeta(1)}\beta(1).$ On the other hand, $\zeta(1)g\zeta^{-1}$ is the identity at $t=1$ and  has gauged $\zeta\cdot\alpha$ to zero since 
$$ \zeta(1)g\zeta^{-1}\cdot (\zeta\cdot\alpha)=\zeta(1)g\cdot\alpha=\zeta(1)\cdot 0=0. $$
Thus, $\phi ( \zeta\cdot (\alpha,\beta))=(g(0)^{-1}\zeta(1)^{-1}, Ad_{\zeta(1)}\beta(1)) $ which is equivalent to $(g(0)^{-1},\beta(1))$ in the space $G^{\Bbb C}\times_{\frak h^c} \frak m^c$. The mapping $\phi$ is well-defined and settles the diffeomorphism. Thus,  $TG^{\Bbb C}/H^{\Bbb C}$ has a complete \hk structure coming from that of $\Phi^{-1}(0)/\mathcal H$ and the restriction of the \hk metric $\kappa$ to $G/H$ is the original $Ad_G$-invariant
inner product $\langle\ ,\ \rangle$ in $\frak g$, {\it cf.} Proposition 6 in \cite{Da-Sw}. 

The \hk\m\ $(TG^{\Bbb C}/H^{\Bbb C},\kappa)$ may be viewed as a \hk structure on the tangent bundle of the \K\m\ $(G^{\Bbb C}/H^{\Bbb C},\kappa|_{G^{\Bbb C}/H^{\Bbb C}}).$

We consider the following $S^1$-action on $\mathcal A$ at \eqref{e5}, for $\theta\in [0, 2\pi)$
\begin{equation}\begin{aligned}\label{e41}
  e^{i\theta}\cdot (T_0,T_1,T_2,T_3):=(T_0,T_1,\cos\theta\ T_2-\sin\theta\ T_3,\sin\theta\ T_2+\cos\theta\ T_3).
\end{aligned}\end{equation}
It is direct to see this $S^1$-action has preserved both $\mu^{-1}(0)$ and  $ \Phi^{-1}(0)$; it also interchanges with any $\mathcal G$-action. Briefly, we'll show this action has preserved the complex structure $I$ and the \K form $\omega_I$ and hence has preserved the \hk metric $\kappa$ of $ \Phi^{-1}(0)/\mathcal H.$ Let $X$ denote the vector field  generated by this $S^1$-action.
\begin{equation}\begin{aligned}\label{e42} 
X (T)=\frac {d}{d\theta}|_{\theta=0}e^{i\theta}\cdot T=(0,0,-T_3, T_2)=-T_3 \frac {\pd}{\pd T_2}+T_2 \frac {\pd}{\pd T_3}.
\end{aligned}\end{equation}

\begin{lem}\label{L10}
The $S^1$-action has preserved the \hk metric $\kappa$; both $L_X\omega_I$  and $L_XI$ have vanished.
\end{lem}
\begin{proof}
\begin{equation}\begin{aligned}\label{e43} 
L_X\omega_I=L_X(dT_2 dT_3)=d(i_X(dT_2 dT_3))=-d(T_2dT_2+T_3dT_3)=0.
\end{aligned}\end{equation}
Applying the formula
\begin{equation}\begin{aligned}\label{e44} 
(L_XI)(Y)=L_X(IY)-I(L_XY)=[X,IY]-I[X,Y]
\end{aligned}\end{equation}
to see $(L_XI)(\frac {\pd}{\pd T_j})=0,\forall j.$ Since both $I$ and $\omega_I$ are preserved by the $S^1$-action, the \hk metric is an isometry with respect to this action.
\end{proof}

The $S^1$-action by scalar multiplication on the fibers is an isometry, by the uniqueness  of Theorem A in \cite{Fe}, we conclude that near the zero section of $TG^{\Bbb C}/H^{\Bbb C}$, the \hk metric $\kappa$ has coincided with the one constructed over the tangent bundle of the \K\m\ $(G^{\Bbb C}/H^{\Bbb C}, \kappa|_{G^{\Bbb C}/H^{\Bbb C}})$ by Feix using the twistor method.

The next step is to have a close look at the  \K\m\ $(G^{\Bbb C}/H^{\Bbb C}, \kappa|_{G^{\Bbb C}/H^{\Bbb C}}).$
Analogous to the $TG^{\Bbb C}$ case, we work in the subspace $\{T_2=T_3=0\}$ to set a diffeomorphism 
\begin{equation}\begin{aligned}\label{e38}
\varphi: {\Phi}^{-1}(0)/ \mathcal H|_{T_2=T_3=0}&\to G\times_{H} \frak m\\
(T_0,T_1)&\to (\xi(0)^{-1},T_1(1))
\end{aligned}\end{equation}
where $\xi\in\mathcal G, \xi(1)=id,\ \xi\cdot T_0=0.$ This $\varphi$ is well-defined as the mapping in \eqref{e37}.
\begin{equation}\begin{aligned}\label{e109}
\varphi^{-1}: G\times_{H} \frak m&\to {\Phi}^{-1}(0)/ \mathcal H|_{T_2=T_3=0};\\ 
(a,v)&\to (-\frac {dh}{dt} h^{-1}, hvh^{-1})
\end{aligned}\end{equation}
for some $h\in\mathcal G$ with the boundary condition $h(0)=a, h(1)\in H.$

The following diffeomorphisms are well-defined.
\begin{equation}\begin{aligned}\label{e17}
 G\times_{H} \frak m\overset {\varphi^{-1}}\longrightarrow&({\Phi}^{-1}(0)/ \mathcal H|_{T_2=T_3=0},I)&&\overset \eta\longrightarrow {\Phi}_c^{-1}(0)/ \mathcal H^c|_{\beta=0}\overset \phi\longrightarrow G^{\Bbb C}\times_{H^{\Bbb C}} 0\\ (\xi(0)^{-1},T_1(1))\longrightarrow
&\;\;\;\;\;\;\;\;\; (T_0,T_1)&&\longrightarrow \alpha=T_0+iT_1\longrightarrow  g(0)^{-1}
\end{aligned}\end{equation}
where $\xi\in\mathcal G$ such that $\xi\cdot T_0=0$ and $\xi(1)=id; g\in \mathcal G^c$ such that $g\cdot \alpha=0$ and $g(1)=id$.
For any $\zeta\in \mathcal H^c$,
 $\zeta(1)g\zeta^{-1}$ is the identity at $t=1$ and  has gauged $\zeta\cdot\alpha$ to zero, $\phi ( \zeta\cdot\alpha)=\zeta(0)g^{-1}(0)\zeta^{-1}(1)= g^{-1}(0)\zeta^{-1}(1)$. The mapping $\phi$ is well-defined since $\zeta(1)\in H^{\Bbb C}$.

Following the same lines  in the $TG^{\Bbb C}$ case in Lemma \ref{L2}, we'll obtain a mapping similar to \eqref{e18}
\begin{equation}\begin{aligned}\label{e39}
 \phi\cdot\eta\cdot\varphi^{-1}: G\times_{H} \frak m\to G^{\Bbb C}/H^{\Bbb C};\ \ \ (a,v)\to a\exp iv H^{\Bbb C}.
\end{aligned}\end{equation}
By  Proposition \ref{L8},
$G^{\Bbb C}/H^{\Bbb C}$  is the adapted complexification on the tangent bundle of the \R\m\ $(G/H,\frak b)$ 
where $\frak b$ is the normal metric in $G/H$ induced from $\langle\ ,\ \rangle$. 

As computed at \eqref{e77}, the \K potential of the \K metric $(\phi\eta\varphi^{-1})^*\kappa|_{G^{\Bbb C}/H^{\Bbb C}}$
at the point $(a,v)\in G\times_{H} \frak m$ is $\frac {|v|^2}2$.
 Thus, $\kappa|_{G^{\Bbb C}/H^{\Bbb C}}$ is the canonical \K metric of the adapted complexification on $T(G/H, \frak b)$.

Let $G$ be a compact Lie group; $H<G$ be a closed subgroup; $g_n$ be the normal metric on $G/H$ induced from a bi-invariant metric $g$ in $G$. Let $\kappa$ be the \hk metric on $TG^{\Bbb C}/H^{\Bbb C}$ associated to $g$ obtained in this section.
We have the following theorem.

\begin{thm}\label{L9}

(I). The adapted complex structure has existed on the whole tangent bundle of $(G/H, g_n)$ with the canonical \K metric $h$. The resulting \K\m\ is denoted as $(T(G/H, g_n), h)$ which is biholomorphically isometric to the \K \m\ $(G^{\Bbb C}/H^{\Bbb C}, \kappa|_{G^c/H^c})$. 

(II).
Taking the tangent bundle  to the \Rm\ $(T(G/H, g_n), h)$, a \hk metric $\kappa$ is obtained in $T(T(G/H, g_n),h)=T(G^{\Bbb C}/H^{\Bbb C}, \kappa|_{G^{\Bbb C}/H^{\Bbb C}})$ with $\kappa|_{G^{\Bbb C}/H^{\Bbb C}}=h; h_{G/H}=g_n$. 
 All the three metrics $g_n, h$ and $\kappa$ are complete.
 
 (III). Though the \hk structure is defined on the whole $T(G^{\Bbb C}/H^{\Bbb C}, \kappa|_{G^{\Bbb C}/H^{\Bbb C}})$, the adapted complex structure can at most be partially defined on it.
 \end{thm}
\begin{proof}
Since $(G^{\Bbb C}/H^{\Bbb C}, \kappa|_{G^{\Bbb C}/H^{\Bbb C}})=(T(G/H, g_n), h)$ is the adapted complexification over the \R\m\ $(G/H, g_n)$, the statement (III) has followed from Proposition \ref{I5}.
\end{proof}

{\bf Remark.} Given a generalized flag \m\ $G^{\Bbb C}/P$,
using Nahm's equations with fast decay at the infinity, Biquard \cite{Biq}
has constructed a complete \hk metric $\kappa$ on  $T^*G^{\Bbb C}/P$.
After this work is completely done, we realized that  $\kappa|_{G^{\Bbb C}/P}$ is the given \K metric in viewing $G^{\Bbb C}/P$ as the orbit $Ad_G(\tau)$ with the Kirillov-Kostant symplectic form. Though this fact was not mentioned in \cite{Biq}, it turns out to be true, {\it cf.} p.7 in  \cite{Bie}.
Thus,  for those non-positively curved $G^{\Bbb C}/P$, there are plenty of them, the adapted complex structure is not defined on the whole tangent bundle whereas a  \hk structure does exist in $T^*G^{\Bbb C}/P$.

 Institute of Mathematics, Academia Sinica

6F, Astronomy-Mathematics Building

No. 1, Sec. 4, Roosevelt Road

Taipei 10617, TAIWAN\\
{\it E-mail address: kan@math.sinica.edu.tw}

\end{document}